\documentclass[a4paper,11pt,twoside,leqno]{article}
\usepackage{amsmath,amssymb,amsfonts,amsthm,graphicx,fancyhdr}
\usepackage[latin1]{inputenc}
\usepackage[T1]{fontenc}
\usepackage{rotating}
\usepackage[colorlinks=true]{hyperref}
\usepackage{suetterl}

\newfont{\suet}{suet14}
\newfont{\schwell}{schwell}

\newtheorem{teo}{Theorem}[section]
\newtheorem{lem}[teo]{Lemma}
\newtheorem{prop}[teo]{Proposition}
\newtheorem{cor}[teo]{Corollary}

 \newtheoremstyle{drem}
      {3pt}
      {3pt}
      {\rmfamily}
      {}
      {\bf}
      {:}
      {.5em}
      {}

 \theoremstyle{drem}

\newtheorem{exa}[teo]{Example}

\newcommand{\eg}[0]{\emph{e.g.} }
\newcommand{\ie}[0]{\emph{i.e.} }

\newcommand{\jo}[1]{\mathcal{#1}}

\newcommand{\pgen}[1]{\langle #1 \rangle}

\newcommand{\eps}[0]{\varepsilon}

\newcommand{\rr}[0]{\ensuremath{\mathbb{R}}}
\newcommand{\cc}[0]{\ensuremath{\mathbb{C}}}
\newcommand{\zz}[0]{\ensuremath{\mathbb{Z}}}

\newcommand{\del}[0]{\ensuremath{\partial}}

\newcommand{\Id}[0]{\mathrm{Id}}

\newcommand{\img}[0]{\mathrm{Im}\,}

\newcommand{\sgn}[0]{\ensuremath{\mathrm{sgn}}}

\newcommand{\maxx}[1]{\textrm{\raisebox{.5ex}{\mbox{$\underset{#1}{\max}$}}} \:}
\newcommand{\minn}[1]{\textrm{\raisebox{.5ex}{\mbox{$\underset{#1}{\min}$}}} \:}

\newcommand{\vide}[0]{\varnothing}

\newcommand{\comp}[0]{\mathsf{c}}

\begin{document}
\centerline{\Large An isoperimetric constant for signed graphs} 

\vspace*{1cm}

\renewcommand*{\thefootnote}{\fnsymbol{footnote}}
\centerline{\large Antoine Gournay\footnotemark} 

\vspace*{1cm}

%
%
%
%

\footnotetext{T.U. Dresden, Fachrichtung Mathematik, Institut für Geometrie, 01062 Dresden, Germany.}

\setcounter{footnote}{0}
\renewcommand*{\thefootnote}{\roman{footnote}}

\begin{abstract}
We introduce a sign in the usual Laplacian on graphs and present the corresponding analogue of the isoperimetric constant for this Laplacian, \ie a geometric quantity which enables to bound from above and below the first eigenvalue. The introduction of the sign in the Laplacian is motivated by the study of $2$-lifts of graphs and of the combinatorial Laplacian in higher degree.
\end{abstract}

\section{Introduction}\label{s-intro}

\subsection{Signed graphs and twisted Laplacian}

Take a graph and assign to each edge $e$ a sign $\sgn e \in \{-1,1\}$. The twisted Laplacian with respect to this signing is the operator $\Delta^\tau: \rr^V \to \rr^V$ defined by \[
\Delta^\tau f(x) = |N(x)| f(x) - \sum_{y \in N(x)} \sgn (x,y) f(y)                                                                                                                                                                                                                                                                                                                                                                        \]
where $N(x)$ is the set of vertices neighbouring $x$ (and $|N(x)|$ is the valency of $x$). The main sources of interest to study such operators come from $2$-lifts and higher-order Laplacian. The main result of this paper is to show lower bound on the spectrum of $\Delta^\tau$.

To each cycle $C$, one may assign $\sgn C = \prod_{e \in C} \sgn e$. A coherent signing (or a coherently signed graph) is a signing such that every cycle has positive sign. \ie there is an even number of negatively signed edges along each cycle. 

\begin{exa}\label{ex-cohsig}
If all edges have $\sgn e = -1$, then the graph is coherently signed if and only if there are no odd cycles. This happens if and only if the graph is bipartite. \hfill $\Diamond$

\end{exa}

Note that putting a negative sign on all edges gives the ``signless Laplacian'', see \cite[\S{}7.8]{CRS}.

Given $S \subset V$, denote $\del S$ to be the set of edges between $S$ and $S^\comp$. Let $e_{mc}(S)$ to be the minimal number of edges that need to be removed so that the graph induced on $S$ is coherently signed. 
Let
\[
\psi(S) = \frac{ |\del S| + 2e_{mc}(S)}{|S|} \qquad \textrm{ and } \qquad \psi(G) = \minn{\vide \neq S \subset G} \psi(S). 
\]
One should probably stress that there is no upper bound on the size of $S$ (\eg $|S| \leq |V|/2$). This comes from the fact that the twisted Laplacian has ``generically'' no kernel, \ie the smallest \emph{possibly zero} eigenvalue is the eigenvalue of interest. The main result of this note is to shown that $\psi(G)$ is the correct analogue to the isoperimetric constant in this situation. Desai \& Rao \cite{DR} already introduced a similar quantity for graphs where the signs are all negative.

By proposition \ref{tbase-p}, this constant is $0$ if and only if the graph has a coherent signing. Note that in the classical case all signs are $+$, hence the signing is always coherent. This is to be expected since $\psi(G)$ bounds above and below the smallest (possibly $0$!) eigenvalue. Some interpretation of $\psi$ can be found at the beginning of \S{}\ref{bili}.

\subsection{Results}

A first (and easy) result (see \S{}\ref{s-basprop}) is that the spectrum of the twisted Laplacian is always $\geq 0$. Some bounds are more naturally expressed in terms of a quantity with an almost cosmetic difference to $\psi$: $\tilde{\psi}(G) = \minn{\vide \neq S \subset G} \frac{ |\del S| + 4 e_{mc}(S)}{|S|}$.

\begin{teo}\label{leteo}
Let $G$ be a signed graph and $\Delta^\tau$ its twisted Laplacian. Let $\mu_1$ be the smallest eigenvalue of $\Delta^\tau$ then 
\[
\frac{\psi(G)^2}{2 d_{max}} \leq d_{max} - \sqrt{ d_{max}^2 - \psi(G)^2 }  \leq \mu_1 \leq \tilde{\psi}(G) \leq 2 \psi(G).
\]
\end{teo}
It would be possible to add a weight to edges. If this is the case, the same result holds but with the quantities $|\del S|$ and $e_{mc}(S)$ counted with weights.


Two applications of this result are discussed. 

The first (see \S{}\ref{bili}) comes from a method introduced by Bilu \& Linial in \cite{BL} to produce expander graphs. They show the problem can be reduced to computing the spectrum of a signed adjacency matrix. In the case where the graph is regular, bounding the largest value of this signed adjacency matrix is equivalent to bounding the smallest value of the twisted Laplacian. Though the conjecture from Bilu \& Linial \cite{BL} has now been solved by Marcus, Spielman \& Srivatava \cite{MSS} (in the bipartite case), it might still be useful to have a geometric interpretation for which signings are ``good'' or ``bad''.

The second (see \S{}\ref{s-hideg}) comes from the Hodge-de Rham Laplacian on forms. Indeed, Mantuano \cite{Man} showed that the spectrum of this operator can be approximated by a combinatorial Laplacian in higher degree (\ie made of higher degree boundary and coboundary operators) of a simplicial complex coming from the \v{C}ech simplex of a sufficiently fine open cover. Hence, giving bounds for the discrete higher degree Laplacian is also a way to give bounds for the Hodge-de Rham Laplacian on forms. Furthermore, it could be an good inspiration to find the proper analogue of the Cheeger constant for forms.


A natural question to ask would be to investigate what is the correct quantity to bound the second smallest eigenvalue (when the smallest eigenvalues is $0$). It seems probable that the correct quantity to look can be done in analogy to Daneshgar \& al. \cite{Dan}: for example, for the second smallest eigenvalue, one should probably look at $\tfrac{1}{2} \minn{S,T} \Big( \psi(S) + \psi(T) \Big)$ where the minimum is over finite disjoint subsets.

{\it Acknowledgements:} The author would like to thank A.~Valette for pointing out corrections and clarifications and B.~Colbois for references to the work of T.~Mantuano and discussions around the Hodge-de Rham Laplacian. 

\section{Basic properties of the twisted Laplacian}\label{s-basprop}

Recall $N(x)$ is the set of neighbours of $x$. The valency of $x$ is denoted $d(x) = |N(x)|$. The maximal valency is $d_{max}:= \maxx{x \in V} d(x)$.

 Let $\nabla^\tau f(x,y) = f(y) - \sgn (x,y) f(x)$ be an operator from functions on vertices to function on (oriented) edges. Then,
\[
\pgen{ \nabla^{*\tau} \delta_{(x,y)} , \delta_z} = \pgen{ \delta_{(x,y)} , \nabla^\tau \delta_z} = \left \lbrace \begin{array}{ll}
1 & \textrm{if } z=y \\
-\sgn (x,y) & \textrm{if } z=x \\
0 & \textrm{else}.
\end{array}
\right. 
\]
Note further that, $\nabla^{*\tau} \nabla^\tau f = 2 \Delta^\tau f$. Indeed,
\[
\begin{array}{rl}
\nabla^{*\tau} \nabla^\tau f (w)  
&= \sum_{x \in N(w)} \nabla^\tau f(x,w) -  \sum_{y \in N(w)} \sgn (w,y) \nabla^\tau f(y,w) \\
&= \sum_{x \in N(w)} ( f(w) - \sgn(x,w) f(x) ) -  \sum_{y \in N(w)} \sgn (w,y) ( f(y) - \sgn(w,y) f(w)) \\
&= 2 \big( d(w) f(w)  - \sum_{x \in N(w)} \sgn(x,w) f(x) \big).
\end{array}
\]

\subsection{Coherent signings}

Let $r_{S,T}(x) = 1$ if $x \in S$, $-1$ if $x \in T$ and $0$ else.
\begin{prop}\label{tbase-p}
$\Delta^\tau$ is positive semi-definite. Assume further $G$ is connected, then the following are equivalent:
\begin{enumerate}\renewcommand{\labelenumi}{\normalfont (\alph{enumi})}\setlength{\itemsep}{0pt} 
\item $\Delta^\tau$ has a kernel; 
\item $G$ is coherently signed;
\item $\psi(G) =0$.
\end{enumerate}
\end{prop}
\begin{proof}
Note that $\pgen{f, \Delta^\tau f} = 2\pgen{f, \nabla^{\tau*} \nabla^\tau f} = 2\pgen{\nabla^{\tau}f,  \nabla^\tau f}= 2 \|\nabla^\tau f \|_2^2 \geq 0$. This shows that $\Delta^\tau$ is positive semi-definite.

If $\Delta^\tau$ has a kernel, this would imply $\|\nabla^\tau f \|_2 =0$ for some non-trivial $f$. This means that $|f|$ is constant on connected components. Furthermore, the function $f$ does not change sign across a positively signed edge and changes sign across negatively signed ones. Since $f \neq 0$ and $G$ is connected, there must be an even number of change of sign along any cycle. This implies $G$ is coherently signed. 

On the other hand if $G$ is coherently signed, then there are two sets $S$ and $T$ with $S \cup T = V$, all edges internal to $S$ or $T$ are positively signed and all edges between $S$ and $T$ are negatively signed. This implies $r_{S,T}$ is in the kernel of $\Delta^\tau$.

If $\psi(G)=0$ then there is a $S \subset V$ with $|\del S|=0$ and $e_{mc}(S)=0$. The first equality implies $S=V$ and the second that $G$ is coherently signed. If $G$ is coherently signed then it is obvious that $\psi(G)=0$. 
\end{proof}
For $S,T \subset V$ subsets of vertices, let $E(S) = E \cap (S \times S)$ be the set of edges with both ends in $S$, $E(S:T) = E \cap (S \times T \cup T \times S)$ be the set of edges with one end in $S$ and the other in $T$,
\[
E^\pm(S) = \{e \in E(S) \mid \sgn e = \pm 1 \} \qquad \text{and} \qquad E^\pm(S:T) = \{e \in E(S:T) \mid \sgn e =\pm 1 \}.  
\]
\begin{lem}
If $G$ is coherently signed, then the set of negatively signed edges is a cut-set.
\end{lem}
\begin{proof}
Assume without loss of generality that $G$ is connected. If $G$ is coherently signed, one can define a function $g$ by picking a random vertex $x_0 \in V$ and $g(x) = (-1)^n$ where $n$ is the number of negatively signed edge on a path from $x_0$ to $x$. Since two different paths differ by cycles, the parity of $n$ (hence the values of $g$) does not depend on the choice of path. One obtains two sets, $S = g^{-1}(1)$ and $S^\comp = T= g^{-1}(-1)$ and the cut-set $\del S$ is exactly given by edges whose sign is negative.
\end{proof}

\begin{prop}\label{treduc-p}
If $G$ is coherently signed, then the spectrum of the (usual) Laplacian is the same as that of the twisted Laplacian. 
\end{prop}
\begin{proof}
If $G$ is coherently signed, then, by the previous lemma, there are subsets $S$ and $T$ with $T=S^\comp$ and $E^-(S) = E^-(T) = E^+(S:T) = \vide$. For a function $f: V \to \rr$, define $\tau f(x) = f(x)$ if $x \in S$ and $= -f(x)$ if $x \in T$. Then $\Delta^\tau f = \tau \Delta \tau f$ but $\tau \circ \tau = \Id$. Since $\Delta$ and $\Delta^\tau$ are conjugate operators, they have the same spectrum. 
\end{proof}

\subsection{Estimates of the first eigenvalue}

Most of the proofs in this subsection are adapted from Desai \& Rao \cite{DR}.
\begin{lem}
Let $S$ and $T$ be disjoint subsets of $V$. Then $\pgen{ r_{S,T} , \Delta^\tau r_{S,T} }\geq 4 e_{mc}(S \cup T) + |\del (S \cup T)|$
\end{lem}
\begin{proof}
This is but a straightforward computation:
\[
\begin{array}{rl}
2 \pgen{ r_{S,T} , \Delta^\tau r_{S,T} } 
  &= \| \nabla^\tau r_{S,T}\|_2^2  \\
  &= 8| E^-(S) |  + 8| E^-(T) | + 2|\del (S \cup T)| + 8 | E^+(S:T)| \\
  & \geq 8 e_{mc}(S \cup T) + 2 |\del (S \cup T)|.
\end{array}
\]
The last line follows from the definition of $e_{mc}$.
\end{proof}

Let $h(S) = \minn{T \subset S} \frac{|\del T|}{|T|}$.
\begin{lem}\label{tlemh-l} 
\emph{(see \cite[Proof of Theorem 7.5.16]{CRS})}
Suppose $S \subsetneq V$, $|V| \geq 4$, $|S| \leq |V|/2$ and $f$ is a function on $V$ satisfying  $f(x)>0$ for all $x \in S$ and $f(x) =0$ otherwise. Then
\[
\| \nabla f \|_2^2 \geq \Big( d_{max} - \sqrt{ d_{max}^2 - h(S)^2 } \Big) \| f\|_2^2 
\]
\end{lem}

\begin{teo}
If $\mu_1$ is the smallest eigenvalue of $\Delta^\tau$ on $G$ then $\mu_1 \leq \tilde{\psi}(G) \leq  2 \psi(G)$.
\end{teo}
\begin{proof}
There exists disjoint sets $S$ and $T$ for which the inequality $|E^-(S)|+|E^-(T)|+ |E^+(S,T)| \geq  e_{mc}(S \cup T)$ is an equality. Consequently, look at 
\[
\frac{ \pgen{ r_{S,T} , \Delta^\tau r_{S,T} } }{ \| r_{S,T} \|_2^2} = \frac{ 4 e_{mc}(S \cup T) + |\del (S \cup T)|}{|S \cup T|} \leq \tilde{\psi}(G) \leq 2 \psi(G). \qedhere
\]
\end{proof}

\begin{teo}
If $\mu_1$ is the smallest eigenvalue of $\Delta^\tau$ on $G$ then $\mu_1 \geq d_{max} - \sqrt{ d_{max}^2 - \psi(G)^2 } $.
\end{teo}
\begin{proof}
Let $f$ be an eigenvector for the eigenvalue $\mu_1$. Let $S = \{ x \in V \mid f(x) >0 \}$ and $T = \{ x \in V \mid f(x) <0 \}$. Of course $S \cup T \neq \vide$. Construct a secondary graph $G' = (V',E')$ as follows. $V'$ is made of two copies of $V$, $V^+$ and $V^-$. $E'$ is as $E$, except for edges $e$ in $E^-(S) \cup E^-(T)$ or $E^+(S:T)$. Such an edge $e=(x,y)$ now doubled as $(x^+,y^-)$ and $(x^-,y^+)$ (where $x^\pm$ and $y^\pm$ are the point corresponding to $x$ and $y$ in $V^\pm$). Let $g:V' \to \rr$ be defined by $g(x^+) = |f(x)|$ and $g(x^-) =0$. Then
\[
\|\nabla g \|^2_{\ell^2(E')} = \sum_{(x',y') \in E'} ( g(y')-g(x'))^2 \leq \| \nabla f \|_{\ell^2(E)}^2 
\]
but $\| g \|_{\ell^2(V')} = \| f \|_{\ell^2(V)}$.

If $E^-(S) \cup E^-(T) \cup E^+(S:T) \neq \vide$ or $S \cup T \neq V$, one may then apply Lemma \ref{tlemh-l} to $g$ on the new (unsigned) graph $G'$. Indeed, for $W \subset S \cup T$ in $V^+$, let $S_1 = W \cap S$ and $T_1 = W \cap T$ then in $G'$, $\del W = 2|E^-(S_1)| +2 |E^-(T_1)| + 2|E^+(S_1:T_1)| + |\del (S_1 \cup T_1)| \geq \psi (S_1 \cup T_1)$. Thus $h(S \cup T) \geq \psi(G)$ and one sees that $\mu_1 \geq d_{max} - \sqrt{ d_{max}^2 - \psi(G)^2 }$.

If $E^-(S) \cup E^-(T) \cup E^+(S:T) = \vide$ and $S \cup T = V$, then $G$ is coherently signed, so there is nothing to prove (as $\mu_1 = \psi(G) =0$).
\end{proof}
Note that, using Taylor expansion, one has $ d_{max} - \sqrt{ d_{max}^2 - \psi(G)^2 } \geq \frac{\psi(G)^2}{2 d_{max}}$.

%
%
%

\subsection{Remarks on $\psi$ and application to $2$-lifts}\label{bili}

The first obvious remark on $\psi(G)$, is that $\psi(G) \leq d_{min}$ (the minimal valency). This is straightforward by taking $S = \{x\}$. In many cases, this can be improved. Let $d_{ave}$ be the average valency of a vertex, \ie
\[
d_{ave} = \frac{1}{|V|} \sum_{x \in V} d(x) = \frac{ 2 |E|}{|V|}.
\]
\begin{lem}\label{spantre}
If $G$ is connected then $\psi(G) \leq \psi(V) \leq d_{ave} - 2 + \frac{2}{|V|}.$
\end{lem}
\begin{proof}
Note that a spanning tree is always coherently oriented (and has $|V|-1$ edges) hence,
\[
\psi(V) = \frac{2 e_{mc}(V)}{|V|} \leq \frac{2 |E| - 2 |V| + 2}{|V|}. \qedhere 
\]
\end{proof}
For $S \subset V$ let $\mathrm{coh}(S)$ be a graph induced on $S$ where, furthermore, a minimal number of edges has been removed so that there is a coherent signing. Use superscripts to distinguish the graph whose valency is in question, \eg $d_{ave}^{\mathrm{coh}(S)}$ is the average degree of $\mathrm{coh}(S)$ and $d^G(s)$ is the degree of the vertex $s$ in $G$.
\begin{lem}\label{avedeg}
$d^{\mathrm{coh}(S)}_{ave}  + \psi(S) = \frac{1}{|S|} \sum_{s \in S} d^G(s)$
\end{lem}
\begin{proof}
Note that
\[
d^{\mathrm{coh}(S)}_{ave} = \frac{1}{|S|} \sum_{s \in S} d^{\mathrm{coh}(S)}(s) = \frac{ \sum_{s \in S} d^G(s) - |\del S| - 2 e_{mc}(S)}{|S|}. \qedhere
\]
\end{proof}
For the rest of this section, we restrict to $\ell$-regular graphs. 

The spectral gap of a connected $\ell$-regular graph is the difference between $\ell$ and $\maxx{\lambda \neq \pm \ell} |\lambda|$ where the maximum runs over all eigenvalues of the adjacency matrix different from $\pm \ell$. Bilu \& Linial \cite{BL} introduced the $2$-lift of a graph as a mean to produce larger graphs without reducing the spectral gap of the adjacency matrix. They show that bounding the spectral radius of the new graph boils down to a bound on maximal (in absolute value) eigenvalue of a signed adjacency matrix: $|\lambda| \leq 2 \sqrt{\ell-1}$. 

Since the graphs of interests are $\ell$-regular, it is equivalent to know the spectrum of the signed Laplacian or of the signed adjacency matrix. More precisely, one has the following corollary:
\begin{cor}\label{sadj}
Assume $A$ is the adjacency matrix of $\ell$-regular graph with signs (that is the entries of $A$ are either $0$ or $\pm 1$). If $\lambda_1$ is the largest eigenvalue of $A$, then
\[
\ell -2\psi(G) \leq \ell - \tilde{\psi}(G) \leq  \lambda_1 \leq \sqrt{\ell^2 - \psi(G)^2}.
\]
\end{cor}
For the proof, simply note that, as an operator $A = \ell \Id - \Delta^\tau$. 

Note that in order to get the (desired) bound $\lambda_1 \leq 2 \sqrt{\ell-1}$ from the above inequality, one needs to have the (quite rare) $\psi(G) \geq \ell-2$. Indeed, this implies that $\psi(V) \geq \ell-2$. However, by lemma \ref{spantre}:
\[
\psi(V) \leq \ell-2 + \frac{2}{|V|}.
\]
This leaves very little room. Such a bound can, for example, be attained if in order to find a coherent signing of the whole graph one needs to remove all edges until there only remains a disjoint union of cycles. Of course, the upper bound from Corollary \ref{sadj} is not sharp for all graphs. Hence the above condition is over-restrictive. 

On the other hand, if $\tilde{\psi}(G) < \ell - 2 \sqrt{\ell-1}$, the the desired bound cannot be attained. To interpret this bound, use lemma \ref{avedeg}:
\[
d_{ave}^{\mathrm{coh}(S)} = \ell - \psi(S)
\]
Hence, as soon as a coherent subgraph $\mathrm{coh}(S)$ of the signed graph has an average degree $>\tfrac{\ell}{2} + \sqrt{\ell-1}$, then the desired bound fails.

Note also that Corollary \ref{sadj} only produces a bound for the largest eigenvalue of the signed adjacency matrix, not for the smallest. When the graph is bipartite, these two eigenvalues are of opposite sign, hence the upper bound on the largest is also a lower bound on the smallest.

\section{The combinatorial Laplacian in higher degree}\label{s-hideg}

\subsection{Basic definitions}

Given a polygonal (usually, simplicial) complex there are two natural operations, the boundary operators $\del_k: \jo{C}_k \to \jo{C}_{k-1}$ on chains and the coboundary operator $\nabla_k: \jo{C}^{k-1} \to \jo{C}^k$ on cochains. To define there operators, one needs to fix an orientation for each cell of the complex. Then
\[
\del_k f(y) = \sum_{x \mid y \in \del x} \delta_{xy} f(x)
\]
where $\delta_{xy}$ is $+1$ if the orientation $x$ induces on $y$ is the same as the orientation of $y$ and $-1$ if not. Likewise
\[
\nabla_{k+1} f(z) = \sum_{x \mid x \in \del z} \delta_{xy}' f(x) 
\]
where $\delta_{xy}'$ is $+1$ if the orientation $z$ induces on $x$ is the same as the orientation of $x$ and $-1$ if not.

When the chains take value in $\rr$ or $\cc$ and the $k$-skeletons are finite (or infinite, and one restricts to $\ell^2$-[co]chains), these two operator are adjoints to each other. There is a natural isomorphism between $\jo{C}_k$ and $\jo{C}^k$ as both identify to functions from the $k$-skeleton to the base field. As Hilbert spaces, this is the isomorphism identifying the space to its dual.

Together these operators define a ``higher degree'' Laplacian by $\Delta_k := \nabla_k \del_k + \del_{k+1} \nabla_{k+1}$. 

Let $A: H_1 \to H_2$ and $B: H_2 \to H_3$ be operators between Hilbert spaces such that $\img A \subset \ker B$. Then,
\[
\begin{array}{rll}
\ker A^* 
&= \{ f \in H_2 \mid A^* f =0 \} \\
&= \{ f \in H_2 \mid \forall g \in H_1, \langle A^* f, g \rangle =0 \} \\
&= \{ f \in H_2 \mid \forall g \in H_1, \langle f, A g \rangle =0 \} 
&= (\img A)^\perp
\end{array}
\]
Furthermore, $\ker (A A^*) = \ker A^*$ since $\langle f, A A^* f \rangle = \langle A^* f, A^* f \rangle = \|A^*f\|^2$. Similarly $\ker B^*B = \ker B$. Thus $(\ker B)^\perp \subset (\img A)^\perp = \ker A^*$. This implies the support of $AA^*$ and $B^*B$ are disjoint and $\ker (AA^* + B^*B) = \ker A^* \cap \ker B = (\img A)^\perp \cap \ker B$.

As mentioned above, putting the $\ell^2$-norm on $\jo{C}_k$ and $\jo{C}^k$, one notes that $\nabla_k = \del_k^*$. This implies that the support of $\nabla_k \del_k$ and $\del_{k+1} \nabla_{k+1}$ are mutually orthogonal and their kernel is exactly $(\img \nabla_k)^\perp \cap \ker \nabla_{k+1}$. This kernel is trivial exactly when the $k^\text{th}$-cohomology is trivial. 

By similar general principles, it is easy to show that the spectrum of $\nabla_k \del_k$ and $ \del_k \nabla_k$ actually coincide.

To see what $\nabla_k \del_k$ looks like, fix an orientation to each element of the $k$-skeleton. Say two element of $X^{[k]}$ are incident if they share a $(k-1)$-cell. If one wishes to avoid multiple edges, it is required they share one such cell unless they are equal. Let $N_{-1}(x) = \{ y \in X^{[k]} \mid y \neq x$ and $y$ is incident with $x \}$. Then
\[
\nabla_k \del_k f (x) = |\del x| f(x) - \sum_{y \in N_{-1}(x) } \eps_{yx} f(y)
\]
where $\del x$ is the set of $(k-1)$-cell in the $k$-cell $x$ and 
\[
\eps_{yx} = \left\{ \begin{array}{ll}
-1 & \text{if } x \text{ and } y \text{ induce the same orientation on their common } (k-1)\text{-cell} \\
+1 & \text{if } x \text{ and } y \text{ induce the opposite orientation on their common } (k-1)\text{-cell} \\ 
\end{array} \right.
\]
The (analogous) formulation for $\del_{k+1} \nabla_{k+1}$ is 
\[
\del_{k+1} \nabla_{k+1} f (x) = |\del^* x| f(x) - \sum_{y \in N_{+1}(x) } \eps_{yx}' f(y)
\]
where $\del^* x$ is the set of $(k+1)$-cell having $x$ as a $k$-cell, $N_{+1}(x)$ is the set of $k$-cells (other than $x$) in the $(k+1)$-cells of $\del^* x$ (perhaps with multiplicity) and 
\[
\eps_{yx}' = \left\{ \begin{array}{ll}
-1 & \text{if } x \text{ and } y \text{ receive the same orientation from their common } (k+1)\text{-cell} \\
+1 & \text{if } x \text{ and } y \text{ receive the opposite orientation from their common } (k+1)\text{-cell} \\ 
\end{array} \right.
\]
Putting all this together (with $\Delta_k = \nabla_k \del_k + \del_{k+1} \nabla_{k+1}$) gives:
\[
\Delta_k f(x) = \Big( |\del x| + |\del^* x| \Big) f(x) - \sum_{y \in N_{-1}(x) } \eps_{yx} f(y) - \sum_{y \in N_{+1}(x) } \eps_{yx}' f(y)
\]
So this is a strange version of the Laplacian, there are two differences. First, there is a ``twist'' by a sign. However, if there is no kernel (\ie the cohomology is trivial in degree $k$), then one can hope to evaluate the smallest non-trivial eigenvalue using Theorem \ref{leteo}. Second, the factor before $f(x)$ does not necessarily fit the number of elements in the sum. There is no direct way out of this. 

\subsection{Application to $\Delta_k$}

Say that the polygonal complex has valency default $\ell$ in degree $k$ if, for any $k$-cell $x$, one has $\ell = |N_{-1}(x)| + |N_{+1}(x)| - |\del x| - |\del^*x|$. If this holds then note that 
\[
(\ell \Id + \Delta_k) f(x) = \Big(  |N_{-1}(x)| + |N_{+1}(x)| \Big) f(x) - \sum_{y \in N_{-1}(x) } \eps_{yx} f(y) - \sum_{y \in N_{+1}(x) } \eps_{yx}' f(y).
\]
This operator can be represented by a twisted Laplacian.\footnote{When valency default fails, one can also try to force it by putting weights on edges. However, one then one looses a direct grasp on the original spectrum.}
\begin{cor}
Let $X$ be polygonal complex with valency default $\ell$ in degree $k$. Let $G$ be the graph whose vertices are the $k$-cells of $X$ and the edges between two vertices are given by their common (bounded) $(k-1)$-cell and common (bounding) $(k+1)$-cell. Put a sign on those edge depending whether the orientations induced are the same or different. Then, the smallest eigenvalue $\mu_1$ of $\Delta_k$ is bounded by
\[
d_{max} - \sqrt{ d_{max}^2 - \psi(G)^2 } -\ell \leq \mu_1 \leq \tilde{\psi}(G) -\ell \leq 2 \psi(G) -\ell.
\]
Furthermore, $\mu_1 = 0$ if and only if $H_k(X) \neq \{0\}$.
\end{cor}
Let us indulge into a few remarks.

It is quite obvious that if one starts from an orientable complex 
and looks at the top dimensional Laplacian, orientability means all the above signs may be chosen to be $+1$. Maximality of the dimension will make the term $\del_{k+1} \nabla_{k+1}$ vanish. In fact, one recovers the Laplacian on the $0$-skeleton of the dual complex. Also, if one looks at the (obviously orientable) $0$-skeleton, then one obtains the ``usual'' Laplacian.
 
If the cells are all of the same type, the ``valency default'' condition can be slightly simplified. For example, in a simplicial complex, $|N_{+1}(x)| = (k+1) |\del^* x|$ and $|\del x| =  k+1$, so that the condition reads $\ell = k( |\del^* x| -1) + |N_{-1}(x)| -1$. In a cubical complex, one has $|N_{+1}(x)| = (2k+1) |\del^* x|$ and $|\del x| =  2k$, so that the condition reads $\ell = 2k( |\del^* x| -1) + |N_{-1}(x)|$. 

As such for simplicial or cubical complexes, having valency default $\ell$ for some $\ell$ is the analogue of constant valency for graphs. Indeed, one needs only to check that the number of neighbours via higher and via lower dimensional cells are constant.

Also, part of the left-hand side can be simplified since $d_{max} -\ell = |\del x|+ |\del^*x|$ for the $k$-cell $x$ such that $|N_{+1}(x)|+|N_{-1}(x)|$ is maximal.

One could think that focusing on only one part (\ie $\nabla_k \del_k$ or $\del_{k+1} \nabla_{k+1}$) of the higher degree Laplacian would be simpler. Indeed, the valency default condition is then more easily verified. However, both operators will have a (usually large) kernel. This forces to study further eigenvalues which are not so easy to bound. As mentioned before, an upper bound seems possible to attain using the methods of Daneshgar \& al. \cite{Dan}.

\subsection{Rayleigh quotient}

Another classical way to obtain upper bounds on eigenvalue is to use the Rayleigh quotient: the smallest eigenvalue of $\Delta_k$ is
\[
\minn{0 \neq f \in C_k} \frac{ \langle f, \Delta_k f \rangle}{\|f\|^2} = \minn{0 \neq f \in C_k} \frac{ \|\del_k f\|^2 + \| \nabla_{k+1} f\|^2}{\|f\|^2}.
\]
Again, this ratio is $\neq 0$ if and only if $H^k(X) = \{0\}$. Unsurprisingly, this ratio measures the default of a given $f \in C_k$ to be without boundary (via $\|\del_k f\|^2$) and its distance to the space of boundaries. 

When $X$ is nice enough, \eg comes from a manifold, a more convincing way of putting this forth is to say that $f$ (seen as an element of $\jo{C}_k$) has a small boundary and the Poincaré dual of $f$ seen as an element of $C^k$ ($\hat{f} \in C_{n-k}$) also has a small boundary\footnote{
Recall that the Poincaré dual $\widehat{X}$ of a (simplicial or cubical) complex $X$ is made of pieces of the barycentric subdivision $B(X)$ of $X$. It is defined by associating to each (oriented) $k$-cell of $X$ a (oriented) $(n-k)$-cell so that the intersection number is $1$. The $(n-k)$-cell is built up of many $(n-k)$-cells in the barycentric subdivision.}. In equation:
\[
\mu_1(\Delta_k) = \minn{0 \neq f \in C_k} \frac{ \|\del_k f\|^2 + \| \del_{n-k} \hat{f}\|^2}{\|f\|^2}.
\]
This description is coherent with the fact that the spectrum of $\Delta_k$ and $\Delta_{n-k}$ are the same in manifolds.

Let us give a simple example to illustrate this.
\begin{exa}
Let $X$ be the cubical complex of a flat torus of dimension $n$, \ie if $Y$ is the universal covering the ``usual'' $K(\zz^n,1)$ (the cube with opposed faces identified) then $X = Y / \Lambda$ where $\Lambda = \oplus_{i=1}^n k_i \zz$ and $k_i \in \zz_{>0}$. Orient all cells with the ``natural'' orientation coming from $\rr^n$. 

Take $f$ to be characteristic function of the loop which goes once around the torus in the $1^\text{st}$-direction and whose other coordinates are all $0$. In other words, $f$ takes value $1$ one the edges of the from $(i,0,\ldots,0)$ to $(i+,0,\ldots,0)$, where $i+ = i+1$ if $i <k_1-1$ and $i+ =0$ if $i = k_1-1$ (and $f$ takes $0$ everywhere else). 

It is straightforward to check that $\del_k f =0$, \ie $f$ represents the class of an orientable subcomplex without boundary. However, $\nabla_{k+1} f \neq 0$. This can be seen because there is a $2$-chain representing a cylinder so that one boundary component will be $f$ and the other will be somewhere (in fact, anywhere as long as one does not close up the cylinder). This $2$-chain will have a boundary which is not orthogonal to $f$. Another way to see it is that the Poincaré dual of $f$ is given by $(n-1)$-cells which are obtained from each other by a perpendicular translation. This clearly has a boundary.

Now, take $g$ to be the $1$-chain taking value $1$ on edges of the form $(i,\vec{a})$ to $(i+,\vec{a})$ where $\vec{a}$ is any $(n-1)$-tuple (so that the given edge belongs to $X$). In other words, $g$ is given by all possible translates of $f$. As such, it still has no boundary. Moreover, its Poincaré dual is made of $(n-1)$-dimensional torii, hence has no boundary. So $g$ belongs to the kernel of $\Delta_1$. 

Repeating this construction for each axis will give an orthogonal basis of the kernel of $\Delta_1$ (corresponding to the usual basis of homology). It is also not too hard to describe the elements of the kernel for other degrees.

Using this orthogonal basis, it is relatively easy to get an upper bound on the first non-trivial eigenvalue. If $f$ is still to be the test function, one needs to assume that $\prod_{i \neq 1} k_i \neq 1$ (otherwise $f= g$, and, hence, is in the kernel). Projecting $f$ perpendicularly to the kernel and using the result as a test function for the Rayleigh quotient gives $2(n-1)/ \Big( 1- 1/ \prod_{i \neq 1} k_i \Big)$. One can change $i\neq 1$ by $i \neq j$ when doing the $j^\text{th}$ direction.  \hfill $\Diamond$
\end{exa}

Recall that,
\[
\| \del_k f \|^2 = \sum_{y \in X^{[k-1]}} \Big( \sum_{x \mid y \in \del x} \delta_{xy} f(x) \Big)^2
\]
where $\delta_{xy}$ is $+1$ if the orientation $x$ induces on $y$ is the same as the orientation of $y$ and $-1$ if not. A similar expression can be written down for $\nabla_{k+1}$.

By using a test function of the type $f= r_{S,T}$ (\ie taking value $1$ on $S$ and value $-1$ on $T$), one can readily give an upper bound to the above Rayleigh quotient. To do so, define the signed graph $G$ as before, and label the edges by the $(k-1)$-cell or the $(k+1)$-cell which gave rise to them. Then note that the contribution to the sum of a term of $\|\del_k f\|^2$ indexed by $y$ comes from edges labelled by $y$ where
\begin{enumerate}\renewcommand{\labelenumi}{\normalfont (\roman{enumi})}\setlength{\itemsep}{0pt} 
\item either the extremities are both $S$ or both in $T$ and the sign is $+$,
\item the extremities are in $S$ and $T$ and the sign is $-$, 
\item or the edge leaves $S \cup T$.
\end{enumerate}
Note that in the first two cases the contribution is (in absolute value) of $2$, whereas in the last case it is of $1$. Denote the set of edges labelled by $y$ which fall in the first two cases by $D(y)$, and the set of those who fall in the last case by $B(y)$. Then $\| \del_k f \|^2 = \sum_{y \in X^{[k-1]}} (2|D(y)|+|B(y)|)^2$. If there is a upper bound $K$ (independent of $y$) on $|\{x \mid y \in \del x\}|$, then
\[
\| \del_k f \|^2 = \sum_{y \in X^{[k-1]}} (2|D(y)|+|B(y)|)^2 \leq 2K \sum_{y \in X^{[k-1]}} 2|D(y)|+|B(y)|
\]
Repeating the argument with $\nabla_{k+1} f$, and assuming that, for the same $K$, one has $|\{ z \mid x \in \del z\}| \leq K$ for any $k$-cell $x$, one finds 
\[
\| \del_k f \|^2 + \|\nabla_{k+1} f \|^2 \leq 2K ( 2|E^-(S)|+ 2|E^-(T)| + 2|E^+(S:T)| + |\del(S \cup T)|. 
\]
Hence, by looking at all $f= r_{S,T}$, one has
\[
\minn{0 \neq f \in C_k} \frac{ \langle f, \Delta_k f \rangle}{\|f\|^2}  \leq 2 K \psi(G).
\]
This indicates that a quantity of the type $\psi(G)$ is, intuitively, still the correct thing to look at, even if the valency default condition does not hold. In the process of making this upper bound, there is of course a lot of potential cancellation neglected. This also explains the much better bound when one actually has the valency default condition.

\end{document}